\documentclass[10pt]{amsart}
\usepackage{amsmath, amssymb}
\usepackage{mathrsfs}
\newcommand{\no}[1]{#1}
\renewcommand{\no}[1]{}
\no{\usepackage{times}\usepackage[subscriptcorrection, slantedGreek, nofontinfo]{mtpro}
\renewcommand{\Delta}{\upDelta}}
\usepackage{color}
\usepackage{enumerate,comment}

\date{\today}
\newtheorem{theorem}{Theorem}[section]
\newtheorem{proposition}{Proposition}[section]
\newtheorem{lemma}{Lemma}[section]

\newtheorem{corollary}{Corollary}[section]

\theoremstyle{remark}
\newtheorem{remark}{Remark}[section]
\newtheorem{remarks}{Remarks}[section]

\newcommand{\bel}{\begin{equation} \label}
\newcommand{\ee}{\end{equation}}

\newcommand{\dd}{\mathrm d}

\newcommand{\R}{{\mathbb R}}
\newcommand{\N}{{\mathbb N}}

\newcommand{\cB}{{\mathcal B}}

\newcommand{\cH}{{\mathcal H}}

\newcommand{\cS}{{\mathcal S}}
\newcommand{\cT}{{\mathcal T}}

\def\beq{\begin{equation}}
\def\eeq{\end{equation}}
\newcommand{\bea}{\begin{eqnarray}}
\newcommand{\eea}{\end{eqnarray}}
\newcommand{\beas}{\begin{eqnarray*}}
\newcommand{\eeas}{\end{eqnarray*}}

\numberwithin{equation}{section}

\title[partial DN map]{Double logarithmic stability in the identification of a scalar potential by a partial elliptic Dirichlet-to-Neumann map}

\author[Mourad Choulli]{Mourad Choulli}
\address{IECL, UMR CNRS 7502, Universit\'e de Lorraine, Ile du Saulcy, 57045 Metz cedex 1, France}
\email{mourad.choulli@univ-lorraine.fr}

\author[Yavar Kian]{Yavar Kian}
\address{Aix Marseille Universit\'e, CNRS, CPT UMR 7332, 13288 Marseille, France \& Universit\'e de Toulon, CNRS, CPT UMR 7332, 83957 La Garde, France}
\email{yavar.kian@univ-amu.fr}

\author[Eric Soccorsi]{Eric Soccorsi}
\address{Aix Marseille Universit\'e, CNRS, CPT UMR 7332, 13288 Marseille, France \& Universit\'e de Toulon, CNRS, CPT UMR 7332, 83957 La Garde, France}
\email{eric.soccorsi@univ-amu.fr}

\begin{document}

\begin{abstract}
We examine the stability issue in the inverse problem of determining a scalar potential appearing in the stationary Schr\"odinger equation in a bounded domain, from a partial elliptic Dirichlet-to-Neumann map. Namely, the Dirichlet data is imposed on the shadowed face of the boundary of the domain and the Neumann data is measured on its illuminated face. We establish a log log stability estimate for the $L^2$-norm (resp. the $H^{-1}$-norm) of $H^t$, for $t>0$, and bounded (resp. $L^2$) potentials.
\end{abstract}

\maketitle

\tableofcontents


\section{Introduction}
\label{section1}

\subsection{Settings and main result}

In the present paper $\Omega$ is a bounded domain of $\mathbb{R}^n$, $n\geq 3$, with $C^2$ boundary $\Gamma$. We denote by $\nu(x)$ the outward unit normal to $\Gamma$ computed at $x \in \Gamma$. For $\xi \in \mathbb{S}^{n-1}$ fixed, we introduce the two following subsets of $\Gamma$
\bel{a0}
\Gamma_\pm (\xi)=\{ x \in \Gamma;\ \pm \xi \cdot \nu(x) >0 \},
\ee
and denote by $F$ (resp. $G$) an open neighborhood of $\Gamma _+(\xi)$ (resp. $\Gamma _-(\xi)$) in $\Gamma$. In what follows $\Gamma_+(\xi)$ (resp. $\Gamma_-(\xi)$) will sometimes be referred to as the $\xi$-shadowed (resp., $\xi$-illuminated) face of $\Gamma$.
Next, given $q \in L^\infty (\Omega)$, real-valued, we consider the unbounded self-adjoint operator $A_q$ in $L^2(\Omega )$, acting on his domain $D(A_q)=H_0^1(\Omega)\cap H^2(\Omega)$, as
$$
A_q=-\Delta +q.
$$
We assume throughout the entire text that $0$ is in the resolvent set of $A_q$ (i.e. $0$ is not in the spectrum of $A_q$) and put
$$
\mathcal{Q}=\{ q\in L^\infty (\Omega;\R);\ 0\ \mbox{is not an eigenvalue of}\ A_q\}.
$$
We establish in Section 2 for any $q \in \mathcal{Q}$ and $g \in H^{-1/2}(\Gamma)$ that the boundary value problem (abbreviated to BVP in the sequel)
\bel{bvp}
\left\{
\begin{array}{ll}
(-\Delta +q)u=0  & \mbox{in}\ \Omega ,
\\
u=g &\mbox{on}\ \Gamma,
\end{array}
\right.
\ee
admits a unique transposition solution $u\in H_\Delta (\Omega )=\{ w\in L^2(\Omega );\ \Delta w\in L^2(\Omega )\}$ and that the so-called Dirichlet-to-Neumann (DN in short) map
\bel{ol1}
\Lambda_q: g \mapsto \partial _\nu u
\ee
is a bounded operator from $H^{-1/2}(\Gamma)$ into $H^{-3/2}(\Gamma)$. For $q_j \in \mathcal{Q}$, $j=1,2$, we denote by $u_j$ the solution to \eqref{bvp} where $q_j$ is substituted for $q$. Since $u=u_1-u_2$ satisfies
$$
\left\{
\begin{array}{ll}
(-\Delta +q_1)u= (q_2-q_1) u_2  &\mbox{in}\ \Omega,
\\
u= 0 &\mbox{on}\ \Gamma,
\end{array}
\right.
$$
and $(q_2-q_1) u_2 \in L^2(\Omega)$, it holds true that $u \in D(A_{q_1})$. Therefore $\partial_\nu u \in H^{1 \slash 2}(\Gamma)$ and 
\bel{ol2}
\Lambda _{q_1,q_2}=\Lambda _{q_1}- \Lambda _{q_2} \in \mathscr{B}(H^{-1/2}(\Gamma),H^{1/2}(\Gamma )),
\ee
hence the operator
\bel{ol3}
 \widetilde{\Lambda}_{q_1,q_2}:g\in H^{-1/2}(\Gamma )\cap \mathscr{E}'(F)\rightarrow  \Lambda _{q_1,q_2}(g)_{|G},
\ee
is bounded from $H^{-1/2}(\Gamma )\cap \mathscr{E}'(F)$, endowed with the norm of $H^{-1/2}(\Gamma)$, into $H^{1/2}(G)$.
We denote by $\| \widetilde{\Lambda}_{q_1,q_2} \|$ the norm of $\widetilde{\Lambda}_{q_1,q_2}$ in $\mathscr{B}(H^{-1/2}(\Gamma )\cap \mathscr{E}'(F),H^{1/2}(G))$.

In the present paper we examine the stability issue in the inverse problem of determining the potential $q \in \mathcal{Q}$ appearing in \eqref{bvp} from the knowledge of $\widetilde{\Lambda}_{q_0,q}$, where $q_0$ is an {\it priori} known suitable potential of $\mathcal{Q}$.

Upon denoting by $B_X$ the unit ball of any Banach space $X$, we may now state the main result of this article as follows.
 
\begin{theorem}
\label{thm-main}
For any $\delta>0$ and $t>0$ we may find two constants $c>0$ and $\tilde{c}>0$, both of them depending only on $\delta$ and $t$, such that we have
\begin{equation}
\label{1.1}
\| q_1-q_2\|_{L^2(\Omega )}\leq c \left( \| \widetilde{\Lambda}_{q_1,q_2}\| + \left|\ln \tilde{c}\left|\ln \| \widetilde{\Lambda}_{q_1,q_2}\| \right|\right|^{-t}\right),
\end{equation}
for any $q_1, q_2 \in \mathcal{Q}\cap \delta B_{L^\infty (\Omega )}$ satisfying $(q_2-q_1)\chi_\Omega \in \delta B_{H^t(\mathbb{R}^n)}$, and
\begin{equation}
\label{1.2}
\| q_1-q_2 \|_{H^{-1}(\Omega )}\leq c \left( \| \widetilde{\Lambda}_{q_1,q_2}\| + \left|\ln \tilde{c}\left|\ln \| \widetilde{\Lambda}_{q_1,q_2}\| \right|\right|^{-1}\right),
\end{equation}
for any $q_1, q_2\in \mathcal{Q}\cap \delta B_{L^2 (\Omega )}$.
 \end{theorem}

Let us now briefly comment on Theorem \ref{thm-main}.
 
\begin{remarks}
\label{remarks1.1}
(a) We suppose in Theorem \ref{thm-main} that $q_j$, $j=1,2$, are real-valued but it is not hard to see that the statement can be adapted at the expense of greater technical difficulties, to the case of complex-valued potentials. Nevertheless, for the sake of clarity, we shall restrict ourselves to real-valued potentials in the remaining part of this text.\\
\noindent (b) For $s>n/2$ and $\epsilon \in (0,s-n/2)$ we recall from the interpolation theorem \cite[Theorem 12.4, page 73]{LM} that $H^{n/2+\epsilon}(\Omega )=[H^s(\Omega ),H^{-1}(\Omega )]_\theta$ with $\theta =(s-(n/2+\epsilon)) \slash (s+1)$. Therefore we have
$$
\|q_1-q_2\|_{L^\infty (\Omega )} \leq  C(s)\|q_1-q_2\|_{H^s(\Omega )}^{1-\theta}\|q_1-q_2\|_{H^{-1}(\Omega )}^\theta  
\leq C(s) \delta^{1-\theta}\|q_1-q_2\|_{H^{-1}(\Omega )}^\theta,
$$
for any $q_1, q_2\in \mathcal{Q}$ such that $q_2 \in q_1+ \delta B_{H^s(\Omega )}$, and some constant $C(s)>0$, depending only on $s$. From this and \eqref{1.2} then follows that
\begin{equation}\label{1.3}
\| q_1-q_2 \|_{L^\infty (\Omega )}\leq c \left( \| \widetilde{\Lambda}_{q_1,q_2}\| + \left|\ln \tilde{c}\left|\ln \| \widetilde{\Lambda}_{q_1,q_2}\| \right|\right|^{-1}\right)^\theta.
\end{equation}
\noindent (c) Fix $t \in (0,+\infty)$. Then, arguing as in the derivation of \eqref{1.3}, we find two positive constants $c$ and $\tilde{c}$ such that the estimate
$$
\| q_1-q_2 \|_{L^2 (\Omega )}\leq c \left( \| \widetilde{\Lambda}_{q_1,q_2}\| + \left|\ln \tilde{c}\left|\ln \| \widetilde{\Lambda}_{q_1,q_2}\| \right|\right|^{-1}\right)^{\frac{t}{t+1}},
$$
holds uniformly in $q_1, q_2 \in \mathcal{Q}$ obeying $q_2 \in q_1+ \delta B_{H^t(\Omega )}$.
However in the particular case where $\partial \Omega$ is $C^{[t]+1}$, we point out that the above estimate is weaker than \eqref{1.1}.
\end{remarks}

\subsection{State of the art and comments}

The celebrated inverse problem of determining $q$ from the knowledge of $\Lambda_q$ was first proposed (in a slightly different setting) by Calder\'on in \cite{C}. The uniqueness issue was treated by Sylvester and Uhlmann in \cite{SU} and a log-type stability estimate was derived by Alessandrini in \cite{A}. As shown by Mandache in \cite{M}, this log-type estimate is optimal. 

All the above mentioned results were obtained with the full data, i.e. when measurements are performed on the whole boundary $\Gamma$. 
Taking the Neumann data on $\Gamma_-(\xi)$, while the Dirichlet data is imposed on the whole boundary $\Gamma$, Bukhgeim and Uhlmann proved in \cite{BU} that partial information of $\Lambda_q$ still determines uniquely the potential. 
Their result was improved by Kenig, Sj\"ostrand and Uhlmann in \cite{KSU} by measuring the Dirichlet data on the shadowed face of $\Gamma$ and the Neumann data on the illuminated one. Moreover a reconstruction result was derived by Nachman and Street in \cite{NS} from the same data as in \cite{KSU}.

Stability estimates with partial data go back to Heck and Wang's article \cite{HW}, where the $L^\infty(\Omega)$-norm of $q$ is log log stably recovered from $\Lambda_q$ with partial Neumann data. The same type of estimate was derived in \cite{CDR2}. Both papers require that the Dirichlet data be known on the whole boundary. This constraint was weakened by Caro, Dos Santos Ferreira and Ruiz in \cite{CDR1}. These authors proved log-log stability of $q$ with respect to a partial DN map associated with Dirichlet (resp. Neumann) data measured on a neighborhood of $\cup_{\xi \in N} \Gamma_-(\xi)$ (resp. $\cup_{\xi \in N} \Gamma_+(\xi)$) where $N$ in a subset of $\mathbb{S}^{n-1}$. Their result, which is similar to \eqref{1.1}, is established for the $L^p$-norm, $p \in (1,+\infty)$, of bounded and $W^{\lambda,p}$-potentials $q$ with $\lambda \in (0,1 \slash p]$. Therefore \eqref{1.2} is valid for a wider class of allowable potentials than in \cite{CDR1}.

The derivation of Theorem \ref{thm-main} relies on complex geometrical optics (CGO in short) solutions to \eqref{bvp} and the Carleman inequality established by Bukhgeim and Uhlmann in \cite{BU}. These are the two main ingredients of the analysis carried out in \cite{HW}, but in contrast to \cite{HW}, we use here this Carleman estimate to construct CGO solutions vanishing on a definite part of the boundary $\Gamma$. 

Notice that usual stability estimates in the inverse problem of determining a potential from the full DN map are of log type, while they are of log log type for partial data. Indeed, it turns out that the low frequencies of the Fourier transform of the potential are bounded uniformly in all directions by the DN map, but that this is no longer the case with the partial data.
This technical issue can be remedied by using the analytic properties of the Fourier transform. The additional log in the stability estimate for the potential may thus be seen as the price to pay for recovering this analytic function by its values in a subdomain, which is an ill-posed problem.
 
\subsection{Outline}
The paper is organized as follows. In Sect. \ref{section2} we introduce the transposition solution associated with the BVP \eqref{bvp} and rigorously define the various DN maps required by the analysis of the inverse problem. Sect. \ref{section3} is devoted to building CGO solutions that vanish on some part of the boundary $\Gamma$. These functions are useful for the proof of Theorem \ref{thm-main}, given in Sect. \ref{section4}.

\section{Transposition solutions}
\label{section2}
In this section we define the transposition solution to the BVP associated with suitable data $(f,g)$,
\bel{a1}
\left\{
\begin{array}{ll}
(-\Delta  +q)u=f\quad &\mbox{in}\ \Omega,
\\
u=g &\mbox{on}\ \Gamma,
\end{array}
\right.
\ee
which play a pivotal role in the analysis of the inverse problem carried out in this paper. To this purpose we start by recalling two useful results for the Hilbert space $H_\Delta (\Omega ) = \{ u\in L^2(\Omega);\ \Delta u\in L^2(\Omega) \}$ endowed with its natural norm
$$
\|u\|_{H_\Delta (\Omega)}=\left(\|u\|_{L^2(\Omega )}^2+ \|\Delta u\|_{L^2(\Omega )}^2\right)^{1/2}.
$$
The first result is the following {\it trace theorem}, borrowed from \cite[Lemma 1.1]{BU}.
\begin{lemma}
\label{lemma2.1}
For $j=0,1$, the trace map
$$
t _ju=\partial ^j_\nu u_{|\Gamma},\ u \in \mathscr{D}(\overline{\Omega }),
$$
extends to a continuous operator, still denoted by $t_j$, from $H_\Delta (\Omega)$ into $H^{-j-1/2}(\Gamma)$. Namely, there exists $c_j>0$, such that the estimate
$$
\|t _ju\|_{H^{-j-1/2}(\Gamma )}\leq c_j \|u\|_{H_\Delta(\Omega )},
$$
holds for every $u \in H_\Delta (\Omega)$.
\end{lemma}

Let us denote by $\langle \cdot ,\cdot\rangle_{j+1/2}$, $j=0,1$, the duality pairing between $H^{j+1/2}(\Gamma)$ and $H^{-j-1/2}(\Gamma)$, where the second argument is conjugated.
Then we have the following {\it generalized Green formula}, which can be found in \cite[Corollary 1.2]{BU}. 
\begin{lemma}
\label{lm2}
Let $q$ be in $L^{\infty}(\Omega)$. Then, for any $u\in H_\Delta (\Omega )$ and $v\in H^2(\Omega )$, we have
$$
\int_\Omega (\Delta -q)u\overline{v}\dd x=\int_\Omega u\overline{(\Delta -q)v}\dd x+\langle t _1 u , t _0 v \rangle_{3/2} - \langle t _0 u , t _1 v \rangle_{1/2}.
$$
\end{lemma}

Let $q \in \mathcal{Q}$. By the usual $H^2$-regularity property for elliptic BVPs (see e.g. \cite[Theorem 5.4, page 165]{LM}), we know that for each $f\in L^2(\Omega )$ and $g\in H^{3/2}(\Gamma )$, there exists a unique solution $\mathcal{S}_q(f,g) \in H^2(\Omega)$ to \eqref{a1}. Moreover the linear operator $\mathcal{S}_q$
is bounded from $L^2(\Omega )\times H^{3/2}(\Gamma)$ into $H^2(\Omega)$, i.e. there exists a constant $C>0$ such that we have
\bel{a2}
\| \mathcal{S}_q(f,g) \|_{H^2(\Omega)} \leq C \left( \| f \|_{L^2(\Omega)} + \| g \|_{H^{3 \slash 2}(\Gamma)} \right). 
\ee
For further reference we put $\mathcal{S}_{q,0}(f)=\mathcal{S}_q(f,0)$ and $\mathcal{S}_{q,1}(g)=\mathcal{S}_q(0,g)$, so we have $\mathcal{S}_q(f,g)=\mathcal{S}_{q,0}(f)+\mathcal{S}_{q,1}(g)$ from the linearity of $\mathcal{S}_q$.

Next, applying Lemma \ref{lm2}, we get for all $(f,g) \in L^2(\Omega) \times H^{3/2}(\Gamma )$ and $v \in \mathcal{H}=H_0^1(\Omega)\cap H^2(\Omega)$ that
\begin{equation}
\label{2.1}
-\int_\Gamma g\partial _\nu \overline{v}\dd \sigma (x)+\int_\Omega f\overline{v}\dd x =\int_\Omega u\overline{(-\Delta + q)v}\dd x,\ u=\mathcal{S}_q(f,g).
\end{equation}
In view of the left hand side of \eqref{2.1} we introduce the following continuous anti-linear form on $\mathcal{H}$
\bel{a2b}
\ell (v)=-\int_\Gamma g \partial _\nu \overline{v}\dd \sigma (x)+\int_\Omega f\overline{v}\dd x,\ v\in \mathcal{H}.
\ee
In light of \eqref{a2}, the operator $L=\ell \circ \mathcal{S}_{q,0}$ is bounded in $L^2(\Omega)$. Further, with reference to \eqref{a2b}, we generalize the definition of the anti-linear form $\ell$ to $(f,g) \in \cH^* \times H^{-1 \slash 2}(\Gamma)$, upon setting
$$
\ell (v)=-\langle g, \partial _\nu v\rangle_{1/2} + \langle f,v\rangle,\ v \in \cH,
$$
where $\langle \cdot ,\cdot\rangle$ denotes the duality pairing between $\mathcal{H}$ and $\mathcal{H}^*$, conjugate linear in its second argument. For any $h \in L^2(\Omega)$, $L(h)=\ell \left( \cS_{q,0}(h) \right)$ satisfies
\begin{equation}
\label{2.1b}
|L(h)|\leq \|g\|_{H^{-1/2}(\Gamma)}\|\partial_\nu \cS_{q,0}(h) \|_{H^{1/2}(\Gamma )}+ \|f\|_{\mathcal{H}^*}\|\cS_{q,0}(h) \|_{\mathcal{H}} \leq C \left(\|g\|_{H^{-1/2}(\Gamma )}+ \|f\|_{\mathcal{H}^\ast}\right) \|h\|_{L^2(\Omega )},
\end{equation}
according to \eqref{a2}. Hence $L$ is a continuous anti-linear form on $L^2(\Omega)$. By Riesz representation theorem, there is a unique vector $\cS_q^t(f,g)\in L^2(\Omega)$ such that
we have
$$ -\langle g,\partial_\nu \cS_{q,0}(h) \rangle_{1/2}+ \langle f, \cS_{q,0}(h) \rangle=\int_\Omega \cS_q^t(f,g) \overline{h} \dd x,\ h \in L^2(\Omega). $$
Bearing in mind that $A_q$ is boundedly invertible in $L^2(\Omega)$ (since $0$ is in the resolvent set of $A_q$) and that $\cS_{q,0}=A_q^{-1}$, we obtain upon taking $h=A_q v$ in above identity, where $v$ is arbitrary in $\cH$, that
\begin{equation}\label{2.2}
-\langle g,\partial_\nu v\rangle_{1/2}+ \langle f,v\rangle=\int_\Omega \cS_q^t(f,g) \overline{(-\Delta +q)v}\dd x,\ v \in \cH. 
\end{equation}
Moreover \eqref{2.1b} entails
\bel{a3}
\| \cS_q^t(f,g) \|_{L^2(\Omega)} \leq C \left(\|g\|_{H^{-1/2}(\Gamma )}+ \|f\|_{\mathcal{H}^\ast}\right). 
\ee
For each $(f,g) \in \cH^* \times H^{-1 \slash 2}(\Gamma)$, $\cS_q^t(f,g)$ will be referred to as the {\it transposition solution} to the BVP \eqref{a1}. As a matter of fact we deduce from \eqref{2.2} that 
\bel{a4}
(-\Delta + q ) \cS_q^t(f,g) = f\ \mbox{in the distributional sense in}\ \Omega. 
\ee
Let us establish now that the transposition solution $\cS_q^t(f,g)$ coincides with the classical $H^2(\Omega)$-solution $\cS_q(f,g)$ to \eqref{a1} in the particular case where $(f,g) \in L^2(\Omega )\times H^{3/2}(\Gamma)$.

\begin{proposition}
\label{proposition2.1}
For any $(f,g)\in L^2(\Omega )\times H^{3/2}(\Gamma )$ we have $\mathcal{S}_q^t(f,g)=\mathcal{S}_q(f,g)$.
\end{proposition}
\begin{proof}
Put $u=\cS_q^t(f,g)$. We have $u \in H_\Delta (\Omega)$ directly from \eqref{a4} hence $t _0 u \in H^{-1/2}(\Gamma)$. Further, given $\varphi \in H^{1/2}(\Gamma)$, we may find $v \in \cH$ such that $t_1 v=\varphi$ by the usual extension theorem (see e.g. \cite[Theorem 8.3, p. 39]{LM}). Applying Lemma \ref{lm2} for such a test function $v$, we find that
\[
\int_\Omega f\overline{v}\dd x=\int_\Omega (-\Delta +q)u\overline{v}\dd x= \int_\Omega u \overline{(-\Delta +q) v}\dd x+\langle t_0u, \varphi \rangle_{1/2}.
\]
From this and \eqref{2.2} then follows that 
$\int_\Omega f \overline{v} \dd x = \int_\Omega f \overline{v} \dd x-\langle g, \varphi \rangle _{1/2}+ \langle t_0 u , t_1 v \rangle_{1/2}$,
which entails
$$
\langle g - t_0 u, \varphi \rangle_{1/2}=0.
$$
Since the above identity holds for any $\varphi \in H^{1/2}(\Gamma )$ we obtain that $t_0 u = g$ and hence $u=g$ on $\Gamma$. This yields the desired result.
\end{proof}

For $g \in H^{-1/2}(\Gamma)$ we put $\cS_{q,1}^t(g)=\cS_q^t(0,g)$. We have $\cS_{q,1}^t(g) \in H_{\Delta}(\Omega)$ by \eqref{a4}, with
$$
\|\mathcal{S}_{q,1}^t(g)\|_{H_\Delta (\Omega)}\leq C \left( 1 + \| q \|_{L^{\infty}(\Omega)} \right) \|g\|_{H^{-1/2}(\Gamma)},
$$
from \eqref{a3}. Hence $\cS_{q,1}^t \in \mathscr{B}(H^{-1/2}(\Gamma ),H_\Delta (\Omega))$ so we get
$$ \Lambda_q= t_1 \circ \cS_{q,1}^t \in  \mathscr{B}(H^{-1/2}(\Gamma),H^{-3/2}(\Gamma)), $$
with the aid of Lemma \ref{lemma2.1}.

Finally we have $\mathcal{S}_q^t(f,g)=\mathcal{S}_q^t(f,0)+\mathcal{S}_q^t(0,g)$, by linearity of $\mathcal{S}_q^t$. 
Put $\mathcal{S}_{q,0}^t(f)=\mathcal{S}_q^t(f,0)$ and $\mathcal{S}_{q,1}^t(g)=\cS_q^t(0,g)$.
Since $\mathcal{S}_{q,0}^t \in \mathscr{B}(\mathcal{H}^\ast ,H_\Delta (\Omega ))$ and $\mathcal{S}_{q,1}^t\in \mathscr{B}(H^{-1/2}(\Gamma ),H_\Delta (\Omega ))$, it follows from Lemma \ref{lemma2.1} that
\[
\Lambda_q = t_1 \circ \mathcal{S}_{q,1}^t \mathscr{B}(H^{-1/2}(\Gamma ), H^{-3/2}(\Gamma)).
\]
Further, as $\mathcal{S}_{q,1}^t(g)=\mathcal{S}_{0,1}^t(g)+\mathcal{S}_{q,0}(f)$, with $f=-q\mathcal{S}_{0,1}^t(g)$, we get that
\[
\Lambda _q =\Lambda _0+R_q,
\]
where the operator$R_q$
\[
R_q : g \mapsto t_1 \mathcal{S}_{q,0} \left( -q\mathcal{S}_{0,1}^t(g) \right)\in H^{1/2}(\Gamma ),
\]
is bounded from $H^{-1/2}(\Gamma)$ into $H^{1/2}(\Gamma)$.
\begin{remark}
\label{remark2.1}
Let us denote by $\Sigma _q$ the (continuous) DN map
\[
\Sigma _q: g\in H^{3/2}(\Gamma )\mapsto t_1 \mathcal{S}_{q,1}(g) \in H^{1/2}(\Gamma ).
\]
Then, since $q$ is real-valued, we get from the Green formula that
\[
\int_\Gamma g\overline{\Sigma _q(h)} \dd \sigma(x)=\int_\Gamma \Sigma_q(g)\overline{h} \dd \sigma(x),\ g,h\in H^{3/2}(\Gamma),
\]
which entails that $\left( \Sigma _q^* \right)_{|H^{3/2}(\Gamma )}=\Sigma _q$. On the other hand, we deduce from Lemma \ref{lm2} that $\Sigma_q^*=\Lambda_q$.
\end{remark}


\section{CGO solutions vanishing on some part of the boundary}
\label{section3}

In this section we build CGO solutions to the Laplace equation appearing in \eqref{bvp}, that vanish on a prescribed part of the boundary $\Gamma$.
This is by means of a suitable Carleman estimate borrowed from \cite{BU}. The corresponding result is as follows.

\begin{proposition}
\label{pr-cgo}
For $\delta>0$ fixed, let $q\in \delta B_{L^\infty (\Omega)}$. Let $\zeta, \eta \in \mathbb{S}^{n-1}$ satisfy $\zeta \cdot \eta =0$ and
fix $\epsilon>0$ so small that 
$\Gamma_-^\epsilon =\Gamma_-^\epsilon (\zeta ) = \{x \in \Gamma;\ \zeta \cdot \nu (x)<-\epsilon \} \neq \emptyset$.
Then there exists $\tau_0=\tau_0(\delta)>0$, such that for any $\tau \geq \tau_0$ we may find
$\psi \in L^2(\Omega)$ obeying 
$\|\psi \|_{L^2(\Omega )}\leq C \tau ^{-1/2}$ for some constant $C>0$ depending only on $\delta$, $\Omega$ and $\epsilon$, and such that
the function $u = e^{\tau (\zeta +i\eta ) \cdot x}(1+\psi) \in H_\Delta(\Omega)$ is solution to the BVP
$$ 
\left\{ \begin{array}{ll} 
(-\Delta +q)u=0 & \mbox{\rm in}\ \Omega, \\
u=0 & \mbox{\rm on}\ \Gamma_-^\epsilon.
\end{array} \right.
$$
\end{proposition}
\begin{proof}
The proof is made of three steps.\\
\noindent {\it Step 1: A Carleman estimate.}
For notational simplicity we write $\Gamma_{\pm}$ instead of $\Gamma_{\pm}(\zeta)$, which isdefined in \eqref{a0}, and recall from \cite{BU} that
we may find two constants $\tau_0=\tau_0(\delta)>0$ and $C=C(\delta)>0$ such the estimate
\beas
& & C\tau ^2\int_\Omega e^{-2\tau x\cdot \zeta}|v|^2\dd x+\tau \int_{\Gamma _+}|\zeta \cdot \nu(x)|e^{-2\tau x\cdot \zeta}|\partial _\nu v|^2\dd \sigma (x)  \\
&\leq & \int_\Omega e^{-2\tau x\cdot \zeta}|(\Delta -q)v|^2\dd x +\tau \int_{\Gamma _-}|\zeta \cdot \nu(x)|e^{-2\tau x\cdot \zeta}|\partial _\nu v|^2\dd \sigma (x),
\eeas
holds for all $\tau \geq \tau_0$ and $v \in \cH$. Since $\Gamma_\pm (-\zeta )=\Gamma_\mp (\zeta)$, the above inequality may be equivalently rewritten as
\bea
& & C\tau ^2\int_\Omega e^{2\tau x\cdot \zeta}|v|^2\dd x+\tau \int_{\Gamma _-}|\zeta \cdot \nu(x)|e^{2\tau x\cdot \zeta}|\partial _\nu v|^2\dd \sigma (x) \label{3.2}
\\
& \leq & \int_\Omega e^{2\tau x\cdot \zeta}|(\Delta -q)v|^2\dd x +\tau \int_{\Gamma _+}|\zeta \cdot \nu(x)|e^{2\tau x\cdot \zeta}|\partial _\nu v|^2\dd \sigma (x),\ v \in \cH,\ \tau \geq\tau_0. \nonumber
\eea
In view of more compact reformulation of \eqref{3.2} we introduce for each real number $\tau$ the two following scalar products:
$$(u,v)_\tau =\int_\Omega e^{2\tau x\cdot \zeta}u\overline{v}\dd x\ \mbox{in}\ L^2(\Omega), $$
and 
$$ \langle \phi ,\psi \rangle_{\tau ,\mu ,\pm} =\int_{\Gamma_\pm}\mu (x)e^{2\tau x\cdot \zeta}\phi\overline{\psi}\dd \sigma (x),\ \mbox{where}\ \mu (x)=|\zeta \cdot \nu(x)|,\ \mbox{in}\ L^2(\Gamma_\pm).$$
We denote by $L_\tau ^2(\Omega)$ (resp., $L_{\tau ,\mu} ^2(\Gamma_\pm)$) the space $L^2(\Omega)$ (resp., $L^2(\Gamma_\pm)$) endowed with the norm $\|\cdot \|_\tau$ (resp. $\|\cdot \|_{\tau ,\mu ,\pm}$) generated by the scalar product $(\cdot,\cdot)_\tau$ (resp., $\langle \cdot ,\cdot \rangle_{\tau ,\mu ,\pm}$).
With these notations, the estimate \eqref{3.2} simply reads
\begin{equation}
\label{3.3}
C\tau ^2\|v\|_\tau ^2+\tau \|\partial _\nu v\|_{\tau ,\mu ,-}^2\leq \|(\Delta -q)v\|^2_\tau +\tau \|\partial _\nu v\|_{\tau ,\mu ,+}^2,\ v \in \cH,\ \tau \geq \tau_0.
\end{equation}

\noindent {\it Step 2: Building suitable transposition solutions.}
Let us identify $L_0^2(\Omega )=L^2(\Omega )$ (resp., $L_{0,1}^2(\Gamma _\pm)=L^2(\Gamma _\pm)$) with its dual space, so the space dual to $L_\tau ^2(\Omega )$ (resp., $L_{\tau ,\mu} ^2(\Gamma _\pm)$)can be identified with $L_{-\tau} ^2(\Omega)$ (resp. $L_{-\tau ,\mu ^{-1}}^2(\Gamma_\pm)$).
Next we consider the operator 
\bel{a5b}
P: v \in \cH \mapsto ((\Delta -q)v,\partial_\nu v_{|\Gamma _+})\in L_\tau ^2(\Omega) \times L_{\tau ,\mu} ^2(\Gamma_+),
\ee
which is injective by \eqref{3.2}. Therefore, for each $(f,g)\in L^2_{-\tau} (\Omega )\times L^2_{-\tau , \mu ^{-1}}(\Gamma _-)$, the following anti-linear form 
\bel{a6}
\ell (w_1,w_2)=(f,v)+\langle g , \partial _\nu v \rangle _- ,\ (w_1,w_2)=P v,\ v \in \cH,
\ee
where $(\cdot ,\cdot )$ (resp., $\langle \cdot ,\cdot \rangle_\pm$) denotes the usual scalar product in $L^2(\Omega )$ (resp., $L^2(\Gamma_\pm )$), is well defined on $\mbox{Ran}(P)$.
Moreover it holds true for every $(w_1,w_2)=P v$, where $v$ is arbitrary in $\cH$, that 
\bea
|\ell (w_1,w_2)| & \leq &  \|f\|_{-\tau}\|v\|_{\tau}+\|g\|_{-\tau ,\mu^{-1},-} \|\partial_\nu v \|_{\tau ,\mu, -} \label{a5} \\
& \leq & \left(  \tau^{-1} \|f\|_{-\tau} + \tau^{-1 \slash 2} \|g\|_{-\tau ,\mu^{-1},-} \right) \left(  \tau^{2} \| v \|_{\tau}^2+ \tau \|\partial_\nu v \|_{\tau ,\mu, -}^2 \right)^{1 \slash 2}. \nonumber
\eea
Thus, upon equipping the space $L_\tau ^2(\Omega )\times L_{\tau ,\mu} ^2(\Gamma _+)$ with the norm
$| \|(w_1,w_2) \| |_\tau =\left( \|w_1\|^2_\tau + \tau\|w_2\|_{\tau ,\mu ,+}^2\right)^{1/2}$, we derive from \eqref{3.3} and \eqref{a5} that
\beas
|\ell (w_1,w_2)|&\leq &  C\left(\tau^{-1}\|f\|_{-\tau}+\tau^{-1/2}\|g\|_{-\tau ,\mu^{-1},-}\right)\left( \|(\Delta -q)v\|^2_\tau +\tau \| \partial_\nu v\|_{\tau ,\mu ,+}^2\right)^{1/2}
\\
&\leq&  C\left(\tau^{-1}\|f\|_{-\tau}+\tau^{-1/2}\|g\|_{-\tau ,\mu^{-1},-}\right) |\|Pv\||_\tau,
\eeas
for some constant $C>0$ depending only on $\delta$. As a consequence we have
\bel{a7}
|\ell (w_1,w_2)| \leq C\left(\tau^{-1}\|f\|_{-\tau}+\tau^{-1/2}\|g\|_{-\tau ,\mu^{-1},-}\right) |\|(w_1,w_2)\||_\tau,\ (w_1,w_2) \in \mbox{Ran}(P).
\ee
Let us identify the dual space of $L_\tau ^2(\Omega )\times L_{\tau ,\mu} ^2(\Gamma _+)$ with $L_{-\tau} ^2(\Omega )\times L_{-\tau ,\mu^{-1}} ^2(\Gamma _+)$ endowed with the norm
$|\|(w_1,w_2)\||_{-\tau} =\left( \|w_1\|^2_{-\tau} + \tau^{-1}\|w_2\|_{-\tau ,\mu{-1} ,+}^2\right)^{1/2}$. Thus, with reference to \eqref{a5b}-\eqref{a6}, we deduce from \eqref{a7} upon applying Hahn-Banach extension theorem that there exists $(v_1,v_2) \in L_{-\tau} ^2(\Omega )\times L_{-\tau ,\mu^{-1}} ^2(\Gamma _+)$ obeying
\begin{equation}
\label{3.4}
| \| (v_1 , v_2 ) \| |_{-\tau} \leq C\left(  \tau^{-1}\|f\|_{-\tau}+\tau^{-1/2}\|g\|_{-\tau ,\mu^{-1},-}\right),
\end{equation}
where $C>0$ is the same as in the right hand side of \eqref{a7}, such that we have
$$
( v_1 , (\Delta -q) v )+ \langle v_2 , \partial _\nu v \rangle_+=(f , v)+\langle g , \partial _\nu v \rangle _-,\  v \in \cH.
$$
Bearing in mind that $f \in \cH^*$ and $g\chi_{\Gamma _-} -v_2\chi_{\Gamma _+} \in H^{-1 \slash 2}(\Gamma)$, where $\chi_{\Gamma_{\pm}}$ is the characteristic function of $\Gamma_{\pm}$
in $\Gamma$, the above identity reads
$$
\langle (-f) , v \rangle - \langle g \chi_{\Gamma _-} -v_2\chi_{\Gamma _+} , \partial _\nu v \rangle_{1 \slash 2} = \int_{\Omega} v_1 \overline{(\Delta -q) v} \dd x,\ v \in \cH.
$$
Therefore, by \eqref{2.2}, $v_1$ is the transposition solution to the following IBVP:
\begin{equation}\label{3.6}
\left\{
\begin{array}{ll}
(-\Delta  +q)v_1=-f &\mbox{in}\ \Omega,
\\
v_1=g\chi_{\Gamma _-} - v_2\chi_{\Gamma _+}& \mbox{on}\ \Gamma.
\end{array}
\right.
\end{equation}

\noindent {\it Step 3: End of the proof.}
Set $\rho =\tau (\zeta +i\eta )$. The last step of the proof involves picking $\varphi \in \mathscr{D}(\mathbb{R}^n;[0,1])$ such that $\varphi(x) =1$ for $x \in \overline{\Gamma_-^\epsilon}$ and $\mbox{supp}(\varphi)\cap \Gamma \subset \Gamma_-^{\epsilon \slash 2}$, and considering the transposition solution $v_1$ to the BVP \eqref{3.6} associated with $f=q e^{\rho \cdot x}$ and $g=-\varphi e^{\rho \cdot x}$. Thus, putting $\psi = e^{-\rho\cdot x} v_1$, we derive from \eqref{3.4} that
\beas
\|\psi \|_0 =\|v_1\|_{-\tau} & \leq & C \left(\tau ^{-1}\|qe^{\rho \cdot x}\|_{-\tau}+\tau^{-1/2}\|\varphi e^{\rho \cdot x}\|_{-\tau ,\mu^{-1},-}\right) 
\\
& \leq &  C \tau^{-1/2} \left( \delta \tau_0(\delta)^{-1 \slash 2} +  \frac{2 | \Gamma_-|}{\epsilon} \right),
\eeas
where $| \Gamma_- |$ denotes the $(n-1)$-dimensional Lebesgue-measure of $\Gamma_-$. Further, bearing in mind that $\rho \cdot \rho =0$, it easy to check that $u=e^{\rho \cdot x}(1+\psi) \in H_\Delta (\Omega)$ satisfies $(-\Delta + q) u =0$ in $\Omega$. Finally, since $v_1=g$ on $\Gamma_-$ by \eqref{3.6}, we obtain that
$\psi = -\varphi = -1$ and consequently that $u=0$ on $\Gamma_-^{\epsilon}$. This proves the desired result.
\end{proof}


\section{Proof of Theorem \ref{thm-main}}
\label{section4}

This section contains the proof of Theorem \ref{thm-main} which consists of a succession of five lemmas.

\subsection{A suitable set of Fourier variables} 
For $\epsilon>0$ we choose $\eta =\eta(\epsilon) \in (0,\pi \slash 2)$ so small that
\bel{s0}
\left( (1-\sin \theta)^2+ 4\cos^2 \theta \right)^{1 \slash 2}<\epsilon,\ \theta \in ( \pi / 2 - \eta, \pi / 2 + \eta ),
\ee
and we define $\cB_\epsilon$ as the set of vectors $\beta=(\beta_1,\ldots,\beta_n) \in \R^n$ with the following spherical coordinates
\begin{equation}
\label{s0b}
\left\{
\begin{array}{lcl}
\beta_1 & = & s \cos \theta_1, \\
\beta_j & = & s \cos \theta_j \left( \prod_{k=1}^{j-1}\sin \theta_k \right),\ j=2,\ldots,n-2, \\
\beta_{n-1} & = & s \sin \phi \left( \prod_{k=1}^{n-2} \sin \theta_k  \right), \\
\beta_n & = & s \cos \phi \left( \prod_{k=1}^{n-2} \sin \theta_k  \right),
\end{array}
\right.
\end{equation}
where $s \in (0,1)$, $\theta_1 \in (\pi / 2 -\eta , \pi /2+\eta)$, $\theta_2 (0, \pi /3)$, $\theta_j \in (0,\pi)$ for $j=3,\ldots,n-2$, and $\phi \in (0, 2\pi)$.
Notice that $\cB_\epsilon$ has positive Lebesgue measure in $\R^n$ given by:
\bel{s1}
|\cB_\epsilon|= \int_0^1 \int_{\pi/2 -\eta}^{\pi /2+\eta}\int_0^{\pi/3}\int_0^\pi \ldots \int_0^\pi \int_0^{2\pi} s^{n-1} \left( \prod_{k=1}^{n-2}\sin ^k \theta _{n-1-k} \right) \dd s \dd \theta_1 \ldots \dd \theta_{n-2} \dd \phi >0.
\ee
For further reference, we now establish the following result.
\begin{lemma}
\label{lm-zeta}
Let $\cT$ be any orthogonal transformation in $\R^n$ that maps $\xi$ onto $e_1=(1,0,\ldots,0)$.
Then for all $\epsilon>0$ and $\kappa \in \cT^* \cB_\epsilon$, there exists $\zeta \in \mathbb{S}^{n-1}$ satisfying
$\kappa \cdot \zeta =0$ and $| \zeta - \xi | < \epsilon$.
\end{lemma}
\begin{proof}
Let $\beta=(\beta_1,\ldots,\beta_n) \in \cB_\epsilon$ be given by \eqref{s0}-\eqref{s0b}. Introduce $\tilde{\zeta} =(\tilde{\zeta}_1,\tilde{\zeta}_2,0,\ldots,0)$, where
$$
\tilde{\zeta}_1=\frac{\sin \theta _1}{\left(\sin^2\theta _1+\frac{\cos^2\theta _1}{\cos^2\theta _2} \right)^{1 / 2}},\ \tilde{\zeta}_2=-\frac{\frac{\cos \theta _1}{\cos\theta _2}}{\left( \sin^2\theta _1+\frac{\cos^2\theta _1}{\cos^2\theta_2} \right)^{1 / 2}}.
$$
Evidently we have $\tilde{\zeta} \in \mathbb{S}^{n-1}$ and $\beta \cdot \tilde{\zeta}=0$. Moreover it follows from \eqref{s0} that
$$
| \tilde{\zeta} -e_1|^2=\frac{(\sin \theta _1-1)^2+\frac{\cos^2 \theta _1}{\cos^2\theta _2}}{\sin^2\theta _1+\frac{\cos^2\theta _1}{\cos^2\theta _2}}\leq (\sin \theta _1-1)^2+4\cos^2 \theta_1<\epsilon^2.
$$
Finally, bearing in mind that $\kappa = \cT^* \beta$ and $e_1=\cT \xi$, we obtain the desired result upon taking $\zeta = \cT^* \tilde{\zeta}$.
\end{proof}

\subsection{Alessandrini's identity and consequence}
For $\epsilon>0$ put
$F^\epsilon(\zeta)=\Gamma \setminus \Gamma_-^\epsilon(\xi)$ and
$G^\epsilon(\xi)=\Gamma \setminus  \Gamma_-^\epsilon(-\xi)$, where $\Gamma_-^\epsilon(\pm\xi)$ is the same as in Proposition \ref{pr-cgo}. Since $F^\epsilon(\xi)=\{ x \in \Gamma;\ \xi \cdot \nu (x) \geq -\epsilon \}$ (resp., $G^\epsilon(\xi)=\{ x\in \Gamma;\ \xi \cdot \nu (x)\leq \epsilon \}$), it holds true that $\cap_{\epsilon > 0} F^\epsilon(\xi) = \{ x \in \Gamma;\ \xi \cdot \nu (x) \geq 0 \} = \overline{\Gamma_+(\xi)}$ (resp., $\cap_{\epsilon > 0} G^\epsilon(\xi) = \{ x \in \Gamma;\ \xi \cdot \nu (x) \leq 0 \} = \overline{\Gamma_-(\xi)}$). Thus, from the very definitions of $F$ and $G$, we may choose
$\epsilon_0=\epsilon_0(\xi,F,G)>0$ so small that 
\bel{pm1}
F^{2 \epsilon}(\xi) \subset F\ \mbox{and}\ G^{2 \epsilon}(\xi) \subset G,\ \epsilon \in (0,\epsilon_0].
\ee
Having said that we turn now to proving the following statement.
\begin{lemma}
\label{lm-ale}
Let $\cT$ be the same as in Lemma \ref{lm-zeta}, let $\tau \in [\tau_0,+\infty)$, where $\tau=\tau_0(\delta)$ is defined in Proposition \ref{pr-cgo}, and let $\epsilon \in [0,\epsilon_0)$.
Then, there exists a constant $C>0$, depending only on $\Omega$, $\delta$ and $\epsilon$, such that the estimate
$$
\left| \int_\Omega (q_2-q_1) e^{-i \kappa \cdot x} \dd x \right| \leq C \left( e^{2 d \tau} \| \widetilde{\Lambda}_{q_1,q_2} \| + \tau ^{-1/2} \right),
$$
holds uniformly in $\kappa \in r \cT^* \cB_\epsilon$ and $r \in (0,2 \tau)$. Here $\| \widetilde{\Lambda}_{q_1,q_2} \|$ is the $\mathscr{B}(H^{-1/2}(\Gamma )\cap \mathscr{E}'(F),H^{1/2}(G))$-norm of the operator $\widetilde{\Lambda}_{q_1,q_2}$ defined in \eqref{ol1}--\eqref{ol3}, and $d=d(\Omega)=\max_{x \in \overline{\Omega}}|x|<\infty$.
\end{lemma}
\begin{proof}
Fix $r \in (0,2 \tau)$, $\kappa \in r \cT^* \cB_\epsilon$, and let $\zeta$ be given by Lemma \ref{lm-zeta}. Pick $\ell \in \R^n$ such that $\ell \cdot \kappa = \ell \cdot \zeta =0$ (which is possible since $n \geq 3$) and $| \kappa+\ell |^2=| \kappa |^2+| \ell |^2 = 4 \tau ^2$. Set
$$ \rho_j = (-1)^j \tau \zeta -i \frac{\kappa + (-1)^j \ell}{2},\ j=1,2, $$
and let $u_j=e^{\rho_j \cdot x}(1+\psi_j) \in H_\Delta (\Omega )$ be defined in accordance with Proposition \ref{pr-cgo}, in such a way that we have
\bel{p0}
\left\{ \begin{array}{ll} 
(-\Delta +q_j) u_j=0 & \mbox{\rm in}\ \Omega, \\
u_j=0 & \mbox{\rm on}\ \Gamma_-^\epsilon((2j-3) \zeta),
\end{array} \right.
\ee
and 
\bel{p0b}
\| \psi_j \|_{L^2(\Omega)}\leq C \tau ^{-1/2}.
\ee
Put $u=S_{q_1,1}^t(t_0 u_2)$ so we have $w=u-u_2=S_{q_1,0}((q_1-q_2) u_2) \in \cH$ and
\bel{4.1}
t_1 w=\Lambda_{q_1}(t_0u_2)-\Lambda_{q_2}(t_0 u_2).
\ee
Upon applying the generalized Green formula of Lemma \ref{lm2} with $u=u_1$ and $v=\overline{w}$, we derive from \eqref{p0} that
\begin{equation}
\label{4.2}
\int_\Omega (q_2-q_1) u_1 u_2 \dd x =\langle t_0 u_1 , t_1 \overline{w} \rangle_{1/2}.
\end{equation}
Notice from the second line of \eqref{p0} that the trace $t_0 u_2$ (resp., $t_0 u_1$) is supported in $F^\epsilon(\zeta)=\Gamma \setminus \Gamma_-^\epsilon(\zeta)$ (resp., $G^\epsilon(\zeta)=\Gamma \setminus \Gamma_-^\epsilon(-\zeta)$), where 
$\Gamma_-^\epsilon(\pm \zeta)$ is defined in Proposition \ref{pr-cgo}. Otherwise stated we have $F^\epsilon(\zeta)=\{ x \in \Gamma;\ \zeta \cdot \nu (x) \geq -\epsilon \}$
(resp., $G^\epsilon(\zeta)=\{ x\in \Gamma;\ \zeta \cdot \nu (x)\leq \epsilon \}$) and hence
$F^\epsilon(\zeta) \subset F^{2\epsilon}(\xi)$ (resp., $G^\epsilon(\zeta)\subset G^{2\epsilon}(\xi)$)
since $|\zeta -\xi|< \epsilon$. From this and \eqref{pm1} then follows that $\mbox{supp} (t_0 u_2) \subset F$ (resp., $\mbox{supp} (t_0 u_1) \subset G$), which together with \eqref{4.1}-\eqref{4.2} yields
\begin{equation}
\label{4.3}
\left| \int_\Omega (q_2 - q_1 )u_1 u_2 \dd x \right| \leq \| \widetilde{\Lambda}_{q_1,q_2} \|\|t _0 u_1\|_{H^{-1/2}(\Gamma )} \|t _0 u_2\|_{H^{-1/2}(\Gamma )}.
\end{equation}
%
Moreover, by \eqref{p0}-\eqref{p0b} and the very definition of $u_j$, $j=1,2$, we get that
\bea
\label{p0c}
\|t_0 u_j\|_{H^{-1/2}(\Gamma)} & \leq & c_j \left( \|u_j\|_{L^2(\Omega)}+\| \Delta u_j\|_{L^2(\Omega)} \right) 
\leq  c_j \left( \|u_j\|_{L^2(\Omega )} + \| q_j u_j \|_{L^2(\Omega )} \right) \\
& \leq & C e^{d \tau} (1+\tau^{-1/2}). \nonumber
\eea
Upon possibly substituting $\max(1, \tau_0(\delta))$ for $\tau_0(\delta)$ (which does obviously not restrict the generality of the above reasoning) we deduce from \eqref{4.3}-\eqref{p0c} that
\begin{equation}
\label{4.4}
\left| \int_\Omega (q_2-q_1) u_1 u_2 \dd x \right| \leq C e^{2 d \tau} \| \widetilde{\Lambda}_{q_1,q_2} \|.
\end{equation}
Now the desired result follows readily from \eqref{p0b} and \eqref{4.4} upon taking into account that $u_1 u_2=e^{-i \kappa \cdot x}(1+\psi_1)(1+\psi_2)$.
\end{proof}

\subsection{Bounding the Fourier coefficients}

We introduce the function $q : \R^n \to \R$ by setting $q=(q_2-q_1)\chi _\Omega$, where $\chi_\Omega$ denotes the characteristic function of $\Omega$ in $\R^n$. We aim to upper bound the Fourier transform $\widehat{q}$ of $q$, on the unit ball $\mathbb{B}$ of $\R^n$, by means of the following direct generalization of \cite[Theorem 4]{AEWZ}
for complex-valued real-analytic functions.
\begin{theorem}
\label{thm-cpl}
Assume that the function $F:2\mathbb{B}^n\rightarrow \mathbb{C}$ is real-analytic and satisfies the condition
$$
|\partial ^\alpha F(\kappa)|\leq K \frac{|\alpha |!}{\rho  ^{|\alpha |}},\ \kappa \in 2\mathbb{B},\ \alpha \in \mathbb{N}^n,
$$
for some $(K,\rho) \in \R_+^* \times (0,1]$.
Then for any measurable set $E \subset \mathbb{B}$ with positive Lebesgue measure, there exist two constants $M=M(\rho ,|E|)>0$ and $\theta =\theta(\rho ,|E|)\in (0,1)$ such that we have
$$
\|F\|_{L^\infty (\mathbb{B})} \leq M K^{1-\theta} \left( \frac{1}{|E|}\int_E |F(\kappa)| \dd \kappa \right)^\theta.
$$
\end{theorem}
The result is as follows. 

\begin{lemma}
\label{lm-fo}
For all $\epsilon \in (0,\epsilon_0]$ there exist two constants $C=C(\Omega,\delta,\epsilon)>0$ and $\theta=\theta(\Omega,\epsilon) \in (0,1)$, such that we have
$$
\left| \widehat{q}(\kappa) \right|\leq C e^{(1-\theta) r} \left(e^{d \tau} \| \widetilde{\Lambda}_{q_1,q_2} \| + \tau ^{-1/2} \right)^\theta,\ 
\kappa \in r \mathbb{B},\ r \in (0,2 \tau),\ \tau \in [\tau_0,+\infty).
$$
\end{lemma}
\begin{proof}
Let $\tau \in [\tau_0,+\infty)$ and $r \in (0, 2 \tau)$ be fixed. By Lemma \ref{lm-ale}, we have
\bel{4.8}
\left| \widehat{q}(r \kappa) \right|\leq C \left (e^{2 d \tau} \| \widetilde{\Lambda}_{q_1,q_2} \|+ \tau ^{-1/2}\right),\ \kappa \in \cT^* \cB_\epsilon.
\ee
In view of \eqref{s1} we apply Theorem \ref{thm-cpl} with $E=\cT^* \cB_\epsilon$ and $F(\kappa)=\widehat{q}(r \kappa)$ for $\kappa \in 2 \mathbb{B}^n$.
Indeed, in this particular case it holds true for any $\alpha=(\alpha_1,\ldots,\alpha_n) \in \N^n$ and $\kappa \in 2 \mathbb{B}^n$ that
$$
\partial^\alpha F(\kappa)=\int_\Omega q(x) (-ir)^{|\alpha|} \left( \prod_{k=1}^n x_k^{\alpha_k} \right) e^{-i r \kappa \cdot x} \dd x_1 \ldots \dd x_n,
$$
whence 
\bel{f1}
|\partial ^\alpha F(\kappa)| \leq \|q\|_{L^1(\Omega )}r^{|\alpha |}d^{|\alpha |} \leq \left( \|q\|_{L^1(\Omega )}\frac{r^{|\alpha |}}{|\alpha |!} \right) \frac{|\alpha |!}{(d^{-1})^{|\alpha |}} \leq \left( \|q\|_{L^1(\Omega )}e^{r}\right ) \frac{|\alpha |!}{(d^{-1})^{|\alpha |}},
\ee
where we recall that $d=d(\Omega)=\max_{x \in \overline{\Omega}} |x|$. Thus, with reference to \eqref{4.8}-\eqref{f1}, we obtain that
\bel{4.10}
\left| \widehat{q}(r \kappa) \right|\leq C e^{(1-\theta)r} \left( e^{d \tau} \| \widetilde{\Lambda}_{q_1,q_2} \| + \tau ^{-1/2}\right)^\theta,\ \kappa \in \mathbb{B},
\ee
which immediately yields the result.
\end{proof}

\subsection{Stability inequalities}
We turn now to proving the stability inequalities \eqref{1.1}-\eqref{1.2}. We start with \eqref{1.1}. For $t>0$ fixed, we assume that $q \in H^t(\mathbb{R}^n)$ and put $M=\|q\|_{H^t(\mathbb{R}^n)}$. By Parseval inequality, it holds true for every $r>0$ that
\bea
\|q\|_{L^2(\Omega )}^2 = \|\widehat{q}\|_{L^2(\R^n)}^2 & = &\int_{|\kappa|\leq r}|\widehat{q}(\kappa)|^2dk+\int_{|\kappa|> r}|\widehat{q}(\kappa)|^2dk \label{4.12}  \\
&\leq & \int_{|\kappa|\leq r}|\widehat{q}(\kappa)|^2dk+\frac{1}{r^{2t}}\int_{|\kappa|> r}|(1+|\kappa|^2)^{t}|\widehat{q}(\kappa)|^2\dd \kappa \nonumber \\
&\leq & \int_{|\kappa|\leq r}|\widehat{q}(\kappa)|^2dk+\frac{M^2}{r^{2t}}, \nonumber
\eea
since $\int_{\R^n}|(1+|\kappa|^2)^{t}|\widehat{q}(\kappa)|^2 \dd \kappa = \| q \|_{H^t(\Omega)}^2$.
From this and Lemma \ref{lm-fo} then follows that
\bel{4.13}
\|q\|_{L^2(\Omega )}^2 \leq Cr^{n}e^{2(1-\theta )r}\left(e^{d\tau} \|\widetilde{\Lambda}_{q_1,q_2}\|+ \tau ^{-1/2}\right)^{2\theta} +\frac{M^2}{r^{2t}},\  r \in (0, 2 \tau), \tau \in [ \tau_0 , +\infty).
\ee
Let us suppose that 
\bel{se0}
\| \widetilde{\Lambda}_{q_1,q_2} \| \in (0,\gamma_0),
\ee
where $\gamma_0=\upsilon(\tau_0)$ and $\upsilon(\tau)=\tau ^{-1/2} e^{-d \tau}$ for $\tau \in (0,+\infty)$. Notice that $\gamma_0 \in (0,1)$ since $\tau_0 \in [1,+\infty)$. Moreover, $\upsilon$ being a strictly decreasing function on $[\tau_0,+\infty)$, there exists a unique $\tau_* \in (\tau_0,+\infty)$ satisfying $\upsilon(\tau_*)= \| \widetilde{\Lambda}_{q_1,q_2} \|$. By elementary computation we find that
$\tau_* = \left( 2 \left| \ln \| \widetilde{\Lambda}_{q_1,q_2} \| \right| -  \ln \tau_* \right) \slash (2 d)$, which entails
$\tau_* < \left| \ln \| \widetilde{\Lambda}_{q_1,q_2} \| \right|$, since $\tau_* \in (1,+\infty)$ and $d \in [1,+\infty)$. As a consequence the real number $\tau_*$ is greater than $\left (2\left| \ln \| \widetilde{\Lambda}_{q_1,q_2} \| \right|-  \ln \left| \ln \| \widetilde{\Lambda}_{q_1,q_2} \| \right| \right) \slash (2 d)$, so we get that
\bel{se2}
\tau_* > \frac{\left| \ln \| \widetilde{\Lambda}_{q_1,q_2} \| \right|}{2 d},
\ee
upon recalling that $\ln x < x $ for all $x \in (0,+\infty)$.
Further, taking $\tau=\tau_*$ in \eqref{4.13} we obtain for each $r \in (0, 2 \tau_*)$ that
$\|q\|_{L^2(\Omega )}^2 \leq C 2^{2 \theta}  r^n e^{2 r} \tau_*^{- \theta} + M^2 \slash r^{2t} \leq  C' e^{(n+2) r} (2d \tau_*)^{- \theta} +M^2 \slash r^{2t}$, where $C'=C 2^{3 \theta} d^\theta$.
This and \eqref{se2} entail
\bel{se3}
\|q\|_{L^2(\Omega )}^2 \leq C' e^{(n+2) r} \left| \ln \| \widetilde{\Lambda}_{q_1,q_2} \| \right|^{-\theta} + \frac{M^2}{r^{2t}},\ r \in (0, 2 \tau_*).
\ee
The next step of the derivation involves finding $r_* \in (0, 2 \tau_*)$ such that both terms $C' e^{(n+2) r_*} \left| \ln \| \widetilde{\Lambda}_{q_1,q_2} \| \right|^{-\theta}$ and $M^2 \slash r_*^{2t}$ appearing in the right hand side of \eqref{se3} are equal. This can be achieved upon assuming in addition to \eqref{se0} that $\| \tilde{\Lambda}_{q_1,q_2} \| \in (0,\gamma_1)$, where $\gamma_1=e^{-\left( \frac{M^2}{(n+2) C'} \right)^{\frac{1}{1+2 t - \theta}}}$, in such a way that we have
\bel{se4}
(n+2) \left| \ln \| \widetilde{\Lambda}_{q_1,q_2} \| \right|^{1+2t-\theta} \geq \frac{M^2}{C'}.
\ee
Indeed, we see from \eqref{se2} that the function $\iota_d : r \mapsto (d r)^{2t} e^{(n+2) d r}$ satisfies
$$
\iota_d(2 \tau_*) > \left| \ln \| \widetilde{\Lambda}_{q_1,q_2} \| \right|^{2t} e^{(n+2) \left| \ln \| \widetilde{\Lambda}_{q_1,q_2} \| \right|} > \left| \ln \| \widetilde{\Lambda}_{q_1,q_2} \| \right|^{2t} + (n+2) \left| \ln \| \widetilde{\Lambda}_{q_1,q_2} \| \right|^{1+2t},
$$
and hence
$\iota_d(2 \tau_*) > ( M^2 \slash C' )  \left| \ln \| \widetilde{\Lambda}_{q_1,q_2} \| \right|^\theta$
according to \eqref{se4}. Thus there exists $r_* \in (0,2 \tau_*)$ such that we have
$\iota_d(r_*)= ( M^2 \slash C') \left| \ln \| \widetilde{\Lambda}_{q_1,q_2} \| \right|^{\theta}$.
This entails that
$(2t+n+2) d r_* \geq 2 t \ln (d r_*) + (n+2) d r_*$ is greater that $\ln \left( ( M^2 \slash C') \left| \ln \| \widetilde{\Lambda}_{q_1,q_2} \| \right|^\theta \right)$, 
which combined with the estimate $\|q\|_{L^2(\Omega )}^2 \leq 2 M^2 \slash r_*^{2t}$, yields
\bel{se5}
\|q\|_{L^2(\Omega )}\leq 2M \left( \frac{2t+n +2}{\theta} \right)^t \left| \ln \left( \left( \frac{M^2}{C'} \right)^{1 \slash \theta} \left| \ln \| \widetilde{\Lambda}_{q_1,q_2} \|\right| \right) \right|^{-t}.
\ee
Summing up, we obtain \eqref{1.1} with $c=2M ( (2t+n +2) \slash \theta )^t$ and $\tilde{c}=(M^2 \slash C')^{1 \slash \theta}$ provided $\| \widetilde{\Lambda}_{q_1,q_2} \|$ is smaller that $\gamma = \min_{j=0,1} \gamma_j$.

On the other hand, in the particular case where $\| \widetilde{\Lambda}_{q_1,q_2} \| \geq \gamma$, we have
$\|q\|_{L^2(\Omega )} \leq (M \slash \gamma) \| \widetilde{\Lambda}_{q_1,q_2} \|$. Putting this together with \eqref{se5} we end up getting \eqref{1.1} upon possibly enlarging $c$.

Finally, we obtain \eqref{1.2} by arguing as in the derivation of \eqref{1.1} upon preliminarily substituting the estimate
\beas
\| q \|_{H^{-1}(\Omega)} & =& \int_{|\kappa|\leq r}|(1+|\kappa|^2)^{-1}|\widehat{q}(\kappa)|^2\dd \kappa + \int_{|\kappa|> r}|(1+|\kappa|^2)^{-1}|\widehat{q}(\kappa)|^2 \dd \kappa  \\
& \leq & \int_{|\kappa|\leq r}|\widehat{q}(\kappa)|^2 \dd \kappa + \frac{1}{r^2} \int_{|\kappa|> r}|\widehat{q}(\kappa)|^2\dd \kappa \\
& \leq &\int_{|\kappa|\leq r} |\widehat{q}(\kappa)|^2 \dd \kappa+\frac{M^2}{r^2}
\eeas
for \eqref{4.12}. This completes the proof of Theorem \ref{thm-cpl}.

\section{Application to conductivity problem}
\label{section5}
In this section we examine the stability issue in the inverse problem of determining the conductivity coefficient $\sigma$ appearing in the system
\begin{equation}
\label{5.1}
\left\{
\begin{array}{ll}
-{\rm div}(\sigma \nabla u)=f & {\rm in}\ \Omega,
\\
u=g &{\rm on}\ \Gamma,
\end{array}
\right.
\end{equation}
from the partial DN map. The strategy is to link this inverse problem to the one studied in the first four sections of this paper and then apply Theorem \ref{thm-main} in order to derive a suitable stability estimate for $\sigma$.

Assume that $\sigma \in W_+^{1,\infty}(\Omega )=\{ c \in W^{1,\infty}(\Omega;\R);\ c(x)\geq c_0\ \mbox{for some}\ c_0>0 \}$.
Then for any $(f,g)\in L^2(\Omega )\times H^{3/2}(\Gamma)$, we know from the standard elliptic theory that \eqref{5.1} admits a unique solution $\mathcal{S}_\sigma (f,g)\in H^2(\Omega)$, and that the linear operator
\[
\mathcal{S}_\sigma : (f,g)\in L^2(\Omega) \times H^{3/2}(\Gamma ) \mapsto \mathcal{S}_\sigma (f,g)\in H^2(\Omega)
\]
is bounded. In the more general case where $(f,g)\in \mathcal{H}^* \times H^{-1/2}(\Gamma)$, we obtain by arguing in the exact same way as in Section \ref{section2} that there exists a unique $u \in L^2(\Omega)$ obeying
\begin{equation}
\label{5.2}
-\int_\Omega u {\rm div}(\sigma \nabla \overline {v}) \dd x = \langle f, v \rangle -\langle g , \sigma \partial _\nu v \rangle_{1/2},\ v\in \mathcal{H}.
\end{equation} 
Such a function $u$ will be referred to as the transposition solution of \eqref{5.1} and will be denoted by 
$\mathcal{S}^t_\sigma(f,g)$.

Let us introduce the Hilbert space $H_{{\rm div}(\sigma \nabla)}(\Omega)=\{ u \in L^2(\Omega),\ {\rm div}(\sigma \nabla u) \in L^2(\Omega) \}$, endowed with the norm
$\|u\|_{H_{{\rm div}(\sigma \nabla)}(\Omega)}=\left( \|u\|_{L^2(\Omega)}^2+ \| {\rm div}(\sigma \nabla u) \|_{L^2(\Omega)}^2 \right)^{1/2}$. By a slight modification of the proof of Lemma \ref{lemma2.1} (e.g. \cite{BU}), the trace map
$$
t_j^\sigma u=\sigma^j \partial_\nu^j u_{|\Gamma},\ u \in \mathscr{D}(\overline{\Omega}),\ j=0,1, 
$$
is extended to a linear continuous operator, still denoted by
$t_j^\sigma$, from $H_{{\rm div}(\sigma \nabla)}(\Omega)$ into $H^{-j-1/2}(\Gamma)$. 
Thus, bearing in mind that $\mathcal{S}^t(0,g) \in H_{\Delta}(\Omega)$ for $g \in H^{-1/2}(\Gamma)$, we see that the DN map
$$ \Lambda_\sigma : g \in H^{-1/2}(\Gamma) \mapsto t_1^\sigma \mathcal{S}^t(0,g) \in H^{-3/2}(\Gamma), $$
is a bounded operator.

Assume that $\sigma \in W_+^{2,\infty}(\Omega )=W^{2,\infty}(\Omega )\cap W_+^{1,\infty}(\Omega )$. Taking into account that
\[
-{\rm div}\left( \sigma\nabla (\sigma^{-1/2}\overline{v}) \right)=\sigma ^{1/2}\left( -\Delta \overline{v}+\sigma^{-1/2} (\Delta \sigma^{1/2}) \overline{v}\right),\ v \in \cH,
\]
we get upon substituting $(0,\sigma^{-1/2} v)$ for $(f,v)$ in \eqref{5.2}, that
$$
\int_\Omega  \sigma ^{1/2} u \overline {(-\Delta  v+q_\sigma v)} \dd x = -\langle \sigma ^{1/2}g , \partial_\nu v \rangle,\ v \in \mathcal{H},
$$
where $q_\sigma =\sigma^{-1/2}\Delta \sigma ^{1/2}$.
As a consequence we have $\sigma ^{1/2}\mathcal{S}_\sigma^t(0,g)=\mathcal{S}_{q_\sigma}^t(0, \sigma ^{1/2}g)$, and hence
$$
\sigma ^{1/2}\mathcal{S}_\sigma^t(0,\sigma ^{-1/2}g)=\mathcal{S}_{q_\sigma}^t(0, g).
$$
From this and the identity $t_1(\sigma ^{1/2} w)= \sigma^{-1 \slash 2} t_1^{\sigma} w + \frac{1}{2}\sigma^{-1 \slash 2} (\partial _\nu \sigma) t_0w$, which is valid for every $w\in H^2(\Omega )$, and generalizes to $w \in H_{\rm{div}(\sigma \nabla)}(\Omega)$ by duality, we get that
\begin{equation}
\label{5.3}
\Lambda_{q_\sigma}=\frac{1}{2}\sigma^{-1} (\partial _\nu \sigma) I + \sigma^{-1/2} \Lambda_\sigma \sigma^{-1/2}.
\end{equation}

Now pick $\sigma_1,\sigma_2\in W_+^{2,\infty}(\Omega )$ such that $\sigma_1=\sigma _2$ on $\Gamma$ and $\partial_\nu \sigma _1=\partial_\nu \sigma_2$ on $F$. Thus, putting $q_j=q_{\sigma_j}$ for $j=1,2$, we deduce from \eqref{5.3} that
\begin{equation}
\label{5.4}
(\Lambda _{q_1}-\Lambda_{q_2})(g)=\sigma_1^{-1/2}(\Lambda _{\sigma _1}-\Lambda _{\sigma _2})(\sigma_1^{-1/2} g),\ g\in H^{-1/2}(\Gamma )\cap \mathscr{E}'(F).
\end{equation}
Let us next introduce
$$
\widetilde{\Lambda}_{\sigma_1,\sigma_2} : g \in H^{-1/2}(\Gamma )\cap \mathscr{E}'(F) \mapsto (\Lambda_{\sigma _1}-\Lambda _{\sigma _2})(g)_{|G} \in H^{1/2}(G).
$$
We notice from \eqref{5.4} that
$$
\widetilde{\Lambda}_{q_1,q_2} g = \sigma_1^{-1/2} \widetilde{\Lambda}_{\sigma_1,\sigma_2}(\sigma_1^{-1/2} g), 
$$
for every $g\in H^{-1/2}(\Gamma)\cap \mathscr{E}'(F)$, provided $\sigma_1=\sigma_2$ on $\Gamma$ and $\partial_\nu \sigma_1=\partial_\nu \sigma_2$ on $F \cap G$. 
In this case, we may find a constant $C>0$, such that we have
\begin{equation}
\label{5.5}
\| \widetilde{\Lambda}_{q _1,q _2}\|\leq C \| \widetilde{\Lambda}_{\sigma_1,\sigma_2}\|,
\end{equation}
where $\|\cdot \|$ still denotes the norm of $\mathscr{B}(H^{-1/2}(\Gamma )\cap \mathscr{E}'(F),H^{1/2}(G))$.
Here we used the fact that the multiplier by $\sigma_1^{-1/2}$ is an isomorphism of $H^{\pm 1/2}(\Gamma)$.

Finally, taking into account that $\phi =\sigma_1^{1/2}-\sigma_2^{1/2}$ is solution to the system
$$ 
\left\{ \begin{array}{ll}
(-\Delta+q_1 ) \phi =\sigma_2^{1/2}(q_2-q_1) & {\rm in}\ \Omega \\ \phi = 0 & {\rm on}\ \Gamma, \end{array} \right. 
$$
we get $\|\phi \|_{L^2(\Omega )}\leq C\|q_2-q_1\|_{H^{-1}(\Omega )}$ upon taking $f=\sigma_2^{1/2}(q_2-q_1)$ and $g=0$ in \eqref{a3}.
Thus, applying Theorem \ref{thm-main} and recalling \eqref{5.5}, we obtain the:

\begin{corollary}
\label{corollary5.1}
Let $F$ and $G$ be the same as in section \ref{section1}, and let $\delta>0$ and $\sigma_0 >0$. Then for any $\sigma_j \in \delta B_{W^{2,\infty}(\Omega )}$, $j=1,2$, obeying $\sigma_j \geq \sigma_0$ and the condition
\bel{5.7}
\sigma_1=\sigma_2\; \rm{on}\ \Gamma\ \mbox{and}\ \partial _\nu \sigma _1=\partial_\nu \sigma_2\ \rm{on}\ F\cap G,
\ee
we may find a constant $C>0$, independent of $\sigma_1$ and $\sigma_2$, such that we have: 
\[
\|\sigma_1-\sigma_2\|_{L^2(\Omega )}\leq C\left( \| \widetilde{\Lambda}_{\sigma _1,\sigma _2}\| + \left|\ln \widetilde{C}\left|\ln \| \widetilde{\Lambda}_{\sigma_1,\sigma_2}\| \right|\right|^{-1}\right).
\]
\end{corollary}

\begin{remark}
\label{remark5.7}
It is not clear how to weaken assumption \eqref{5.7} in Corollary \ref{corollary5.1}. Indeed, to our knowledge, the best available result in the mathematical literature (this is a byproduct of \cite[Theorem 2.2, p. 922 and Theorem 2.4, p. 923]{AG}) on the recovery of the conductivity at the boundary, claims for any non empty open subset $\Gamma_0$ of $\Gamma$, and for all $\sigma_m$, $m=1,2$, taken as in Corollary \ref{corollary5.1} and satisfying the condition ${\rm supp}\left(\sigma _1-\sigma _2\right)_{|\Gamma}\subset \Gamma _0$ instead \eqref{5.7}, that
$$
\| \partial_\nu^j \sigma_1-\partial_\nu^j \sigma _2 \|_{L^\infty (\Gamma)} \leq C \| \Lambda^0_{\sigma_1,\sigma_2}\|^{1/(1+j)},\ j=0,1.
$$
Here $\Lambda ^0_{\sigma_1,\sigma_2}$ denotes the operator
$g\in H^{-1/2}(\Gamma )\cap \mathscr{E}'(\Gamma_0) \mapsto  (\Lambda_{\sigma_1}-\Lambda_{\sigma_2})(g)_{|\Gamma_0}\in H^{1/2}(\Gamma_0)$,
$\| \Lambda ^0_{\sigma _1,\sigma _2}\|$ is the norm of $\Lambda ^0_{\sigma_1,\sigma_2}$ in $\mathscr{B}(H^{-1/2}(\Gamma)\cap \mathscr{E}'(\Gamma_0 ),H^{1/2}(\Gamma_0))$ and $C>0$ is a constant that depends neither on $\sigma_1$ nor $\sigma_2$. 
\end{remark}


\bigskip
\vskip 1cm

\end{document}